\documentclass[preprint,12pt]{elsarticle}
\usepackage[utf8]{inputenc}
\usepackage[T1]{fontenc}
\usepackage{geometry, comment}
\usepackage{amsmath, amsfonts, amssymb, amsthm}
\usepackage{mathtools}
\usepackage[colorlinks=true,linkcolor=black,citecolor=blue]{hyperref}
\usepackage[dvipsnames]{xcolor}

\theoremstyle{plain}

\newtheorem{lem}{Lemma}[section]
\newtheorem{thm}{Theorem}

\newtheorem{prop}{Proposition}

\newtheorem{cor}{Corollary}[section]
\newtheorem*{thm*}{Theorem}

\theoremstyle{definition}
\newtheorem{defn}{Definition}

\theoremstyle{remark}
\newtheorem*{rem}{Remark}

\newcommand{\N}{\mathbb{N}}

\newcommand{\R}{\mathbb{R}}
\newcommand{\C}{\mathbb{C}}
\newcommand{\Cc}{\mathcal{C}^{\infty}(I^2)}
\newcommand{\Cb}{\overline{\mathcal{C}^{\infty}(I^2)}}

\newcommand{\D}{\mathcal{D}} 

\makeatletter
\newcommand*\bigcdot{\mathpalette\bigcdot@{.5}}
\newcommand*\bigcdot@[2]{\mathbin{\vcenter{\hbox{\scalebox{#2}{$\m@th#1\bullet$}}}}}
\makeatother

\makeatletter
\newcommand\ackname{Acknowledgements.}
  \newenvironment{acknow}{%
      \if@twocolumn
        \section*{\abstractname}%
      \else
        \small
          {\bfseries \ackname}%
      \fi}
      {\if@twocolumn\else\endquotation\fi}
\makeatother

\journal{Journal of Mathematical Analysis and Applications}

\begin{document}

\begin{frontmatter}

\title{A Fr\'{e}chet Lie group on distributions}
\author[1]{Manon Ryckebusch\footnote{Corresponding author, \texttt{manon.ryckebusch@univ-littoral.fr}}}
\author[1]{Abderrahman Bouhamidi}
\author[1]{Pierre-Louis Giscard}

\affiliation[1]{organization={Univ. Littoral C\^{o}te d'Opale, UR 2597, LMPA, Laboratoire de Math\'ematiques Pures et Appliqu\'ees Joseph Liouville}, city={Calais}, postcode={62100}, country={France}}

\begin{abstract}
Solving non-autonomous systems of ordinary differential equations leads to consider a new product of bivariate distributions called the $\star$~product in the literature.
	This product, distinct from the convolution product, has recently been used to establish structural results concerning non-autonomous differential systems, yet its formal underpinnings remain unclear.
 We demonstrate that it is well-defined on the weak closure of the space of smooth functions on a compact subset of $\mathbb{R}^2$. We establish that a subset of this weak closure has the structure of a Fr\'{e}chet space $\mathcal{D}$. The $\star$~product arises from the composition of endomorphisms of that space. Invertible elements of $\mathcal{D}$ form a dense subset of it and a Fr\'{e}chet Lie group for the operation $\star$. This product generalizes the convolution, Volterra compositions of first and second type and induces Schwartz's bracket.
 \end{abstract}


\end{frontmatter}

%\normalsize

%\date{\today}
%\begin{document}
%\maketitle

%\keywords{Distributions, $\star$~product, Fr\'{e}chet Lie groups.}

%\msc{46F10}

	\section{Introduction}
    Laurent Schwartz introduced the theory of distributions in his landmark book on the subject, defining them as linear forms acting on spaces of functions \cite{schwartz1978}. The origins of this theory can be traced to the work of pioneers such as Heaviside \cite{Heaviside1893, Heaviside1893b}, Volterra \cite{Volterra1916, volterra1924}, and others \cite{Peres1,Peres2,Peres3, Mikusiski1948SurLM, Mikusinski1950}, whose contributions laid the groundwork in the decades preceding Schwartz’s breakthrough. Meanwhile, Mikusi\'nski developed an alternative, sequence-based approach to distributions, defining them as limits of sequences of integral operators acting on functions, and relying on weak convergence \cite{Mikusiski1948SurLM}.

A major advantage of Mikusi\'nski's approach is its applicability beyond function spaces, making it possible to define a composition on more general spaces. This approach thus offers a potential route to defining a product on distributions—an issue first raised by Schwartz, who observed that such a product could not be meaningfully defined without violating one or more of the intuitive assumptions about what the product should entail \cite{MR0064324}. These assumptions include the requirement that the constant function 1 be the unit element, that the product must align with the classical product of functions, and that the set on which the product is defined should include the space of continuous real functions of a real variable as a linear subspace.

Several tools have been proposed to resolve or circumvent these issues, including Colombeau algebras \cite{Colombeau}, tensor products, and the convolution of distributions. Notably, certain special Colombeau algebras can be viewed as an extension of Mikusi\'nski's sequential approach to distributions \cite{SpecialCo}. Mikusi\'nski himself developed algebraic structures for operators acting on spaces \cite{Mikusinski1950}, which have since been generalized \cite{MikOperationalCalc}. An important earlier theory, stemming from the study of integral equations by Fredholm, laid the foundation for the modern concept of Fredholm operators \cite{murphy2014c}.

Returning to Mikusi\'nski's foundational work, we demonstrate that a compositional product of distributions in two variables exists. In specific cases, this product, denoted by a $\star$, reduces to convolution; in others, it corresponds to Volterra composition; and in yet others, it induces the classical product of smooth functions of a single variable. This product also relates to Schwartz's bracket and can be viewed as a continuous version of the matrix product for infinite triangular matrices. Restrictions of this product to piecewise smooth functions have been rediscovered and applied in physics \cite{Champs, Giscard2015}. However, in none of these cases has the mathematical well-definedness of the product been rigorously verified, nor have its properties been fully explored.

Despite this, applications of the $\star$ product are already in use. For instance, it has been employed in path sum methods for solving Volterra equations \cite{GiscardVolt}, in the context of non-autonomous differential equations in physics \cite{GiscardTamar}, and in coupled systems of such equations, particularly in the study of Hamiltonians from NMR spin systems in chemistry \cite{NMR, Foroozandeh, bonhomme2023newfastnumericalmethod}. In turn, the numerical simulation of such systems has led to the development of $\star$-product-based Lanczos methods \cite{Giscard_2022, Cipolla2023, GiscardPozza20201}, numerical approximations of $\star$ products \cite{pozza2024, Pozza2023_2, PozzaBuggenhout, Pozza2022ANM, pozza2023new, DD21}, and the first theoretical investigations of the product and its inverses \cite{GiscardPozzaConvergence, GiscardPozza20201, Pozza2023}, culminating in the present work.

In this context, we provide the first rigorous investigation of the topological and algebraic structures associated with this product. Specifically, we explore a Fr\'{e}chet space of distributions and its relationship with the Fr\'{e}chet space of smooth functions, in line with Schwartz's work \cite{schwartz1978}. We also investigate a Lie group structure, analogous to the one that has been studied for the convolution product in linear algebraic groups \cite{jantzen2003}.

The article is structured as follows: In Section \ref{StarDef}, we demonstrate that the $\star$-product is well-defined on the weak closure of the space of smooth functions over a compact subset of $\mathbb{R}^2$, as per Mikusi\'nski’s framework. We then relate this product to existing products. In Section \ref{FLgroup}, we show that the $\star$ product induces a Fr\'{e}chet Lie group structure on a specific set of distributions.

   \section{The $\star$~product} \label{StarDef}
\subsection{Context and definitions}
	In their works on integral equations and permutable functions, Volterra \cite{Volterra1916} and Volterra and P\'er\`es \cite{volterra1924} defined a product of two smooth functions $\tilde{f}$ and $\tilde{g}$, now called the Volterra composition of the first kind, as follows:
\begin{equation}\label{VolterraCompoDef}
 \left(\tilde{f} \star_V \tilde{g}\right)(x,y) := \int_{y}^{x} \tilde{f}(x,\tau) \tilde{g}(\tau,y) d\tau.
 \end{equation}
 This product emerges naturally from the Picard iteration solution to Volterra integral equations \cite{Picard1893} and is frequently rediscovered from there, see e.g. \cite{Champs}, while a partial extension to distributions first appeared in a mathematical-physics context \cite{Giscard2015, GiscardPozza20201}. Working in the 1910s and 1920s, before the advent of distribution and Dirac seminal works, Volterra and P\'er\`es noted that defined as in Eq.~(\ref{VolterraCompoDef}), the product lacks a unit element and suffers from subsequent issues regarding inversion and more \cite{volterra1924}. Upon inspection of Eq.~(\ref{VolterraCompoDef}) it is intuitive to remedy the problem by proposing the Dirac delta distribution as unit for the Volterra composition. In order to formalize this observation, we are forced to consider a more general product.
    \begin{defn}
    Let $I\subset \mathbb{R}$ be compact. Define the $\star_I$ product of $f$ with $g$ by 
	\begin{equation}\label{DefStarI}
	\left(f \star_I g \right) (x,y) := \int_I f(x,\tau) g(\tau,y) d\tau,
	\end{equation}
    where  
    and $f$ and $g$ might be more general objects than smooth functions, objects which we identify precisely in Theorem~\ref{WellDefined}.
    \end{defn}
     In particular we note that for $f(x,y)=\tilde{f}(x,y)\Theta(x-y)$ and $g(x,y)=\tilde{g}(x,y)\Theta(x-y)$ with $\Theta(.)$ the Heaviside Theta function under the convention $\Theta(0)=1$\label{Theta0} and $\tilde{f}$ and $\tilde{g}$ smooth functions over $I^2$, then provided $[y,x]\subset I$,
    \begin{equation}\label{Volt1}
	\left(f \star_I g \right) (x,y) = \int_y^x \tilde{f}(x,\tau) \tilde{g}(\tau,y) d\tau\,\Theta(x-y)=(\tilde{f}\star_V\tilde{g})(x,y)~\Theta(x-y),
	\end{equation}
    that is, we recover the Volterra composition of the smooth functions $\tilde{f}$ and $\tilde{g}$.
    The first task is to identify the objects on which the $\star_I$~product is well defined. 
	
\subsection{The $\star_I$~product is well-defined on $\Cb$}\label{WellDefSec}
 \begin{defn}\label{DefCc}
  Let $I$ be a compact subspace of $\mathbb{R}$. We write $\mathcal{C}^\infty(I^2)$ for the set of functions which are defined, continuous and all of whose derivatives are defined and continuous on an open set $\Omega\subset \mathbb{R}^2$ with $I^2\subset \Omega$.
 \end{defn}
 \begin{thm}\label{WellDefined}
The $\star_I$~product is well defined on the weak closure (see \ref{weakconvDef} below) $\Cb$ of $\Cc$, that is, for any $f,g\in \Cb$, $f\star_I g$ exists and is in $\Cb$. Furthermore the $\star_I$~product is associative over $\Cb$.
 \end{thm}
 We can be more explicit as to which objects $\Cb$ comprises. 
\begin{defn}\label{DefD} Let $\D$ be the set of objects $d$ of the form 
	\begin{equation*} 
d(x,y)=\tilde{d}(x,y)\Theta(x-y)+\sum_{i=0}^{+\infty} \tilde{d}_{i}(x,y) \delta^{(i)}(x-y),
	\end{equation*}
	where $(x,y)\in I^2$, $\tilde{d},\,\tilde{d}_{i}\in \Cc$ are complex-valued functions and $\delta^{(i)}(x-y)$ is the $i$th Dirac delta derivative in the sense of distributions evaluated in $x-y$. 
 \end{defn}
\begin{rem}
The definition of $\mathcal{D}$ given above allows for infinitely many nonzero terms to be present in the sum. The relevant notion of convergence is presented in Section~\ref{FrechetD}.
\end{rem}
 \begin{cor}\label{DinCb}
We have $\D\subset \Cb$ and for $x,y,\in I$, $\delta(x-y)$ is the unit of the $\star_I$ product.
 \end{cor}
Theorem~\ref{WellDefined} together with the above corollary imply a convenient result for $\D$.
\begin{cor}\label{CorStarR}
The $\star_\mathbb{R}$~product is a well-defined associative product on $\D$, that is, for any $d,e\in \D$, 
\[
(d\star_\mathbb{R} e)(x,y):=\int_{-\infty}^{\infty}d(x,\tau)e(\tau,y)d\tau,
\]
is an element of $\mathcal{D}$.
\end{cor}

\begin{rem}
Elements of $\mathcal{D}$ can be called ``distributions'' but with some care as to what is meant. Lemmas~\ref{intern}, \ref{lemTOT} and \ref{lemReg} presented below establish that the $\star_I$~product corresponds to the composition of endomorphisms of $\Cb$ (see also Theorem~\ref{AutThm}). Elements of $\mathcal{D}$ are such endomorphisms and not linear forms, that is, in that sense, not distributions. 
At the same time, we show in Section~\ref{InnOut} that the $\star_I$~product induces the one-dimensional bracket defining the action of linear forms $\mathcal{C}^\infty(I)\to \mathbb{C}$. Similarly the Schwartz bracket on two variables is induced by a four-dimensional $\star_I$~product of endomorphisms of the weak closure of $\mathcal{C}^\infty(\mathbb{R}^4)$. This allows one to use elements of $\mathcal{D}$ as linear forms $\mathcal{C}^\infty(I^2)\to \mathbb{C}$ and so legitimately call them distributions. This umbral construction relating endomorphisms and linear forms extends to higher dimensions but is beyond the scope of the present work. When formally seeing elements of $\mathcal{D}$ as distributions, a natural candidate for the open set is the smallest open set on which all $\tilde{d}$, $\tilde{d}_i$ are defined (and smooth). 
\end{rem} 
       To prove the theorem and its corollaries, we have to delve back into the sequential approach to distributions, showing first that the $\star_I$~product is well-behaved (in ways made precise below) on $\Cc\times \Cc$ and second, that it can be exchanged with limits in products of sequences of elements of $\Cc$. In doing so we broadly follow and adapt the strategy designed by Mikusi\'nski in \cite{Mikusiski1948SurLM} to define certain operations on distributions.  
   \begin{proof}
   We begin with the necessary definitions.
	\begin{defn}Let $F$ be a set and let $\bigcdot$ be a map $F \times F \to F$, called a composition law on $F$.\footnote{In Mikusi\'nski's work such a composition law is called an internal composition law on $F$.} The set $F$ is total with respect to $\bigcdot$ if 
		$f \bigcdot \varphi$ $=$ $g \bigcdot \varphi $ for all $\varphi\in F$ implies $f=g$ for any $f$ and $g$ $\in F$.
	\end{defn}
	
We now need a notion of convergence on the set $F$. Mikusi\'nski's  original requirement was that one disposes of a map that sends certain sequences $(f_n)$ of elements of $F$ to an element $f$ of $F$, denoted $\lim f_n =f$. Following modern terminology, we shall ask that $F$ be given a convergence structure \cite[Def. 1.1.1; p. 2]{beattie2010convergence} and write $\lim$ to denote the limit of a convergent sequence of elements of $F$.

 \begin{defn}\label{weakconvDef}
Let $F$ be a set with a convergence structure, and let $\bigcdot$ be a  composition law on $F$ for which $F$ is total. The weak limit $f$ 
 of a sequence $(f_n)$ of elements of $F$, is defined as the set of all sequences of $(f_n')$ of elements of $F$ for which
\[
\forall \varphi\in F,~\lim f_n \bigcdot \varphi=~\lim f'_n \bigcdot \varphi,
\]
the limits being taken with respect to the convergence on $F$, i.e. assuming $\lim f_n \bigcdot \varphi \in F$. We define the composition of a weak limit with an element $\varphi\in F$ by
\( 
f\bigcdot\varphi:=\lim f_n\bigcdot\varphi
\), then by construction $f \bigcdot \varphi \in F$ and we  write $\widetilde{\lim}\, f_n=f$. With these constructions, we say that $F$ is endowed with the notion of weak convergence.
\end{defn}

\begin{defn}\label{regulardef}
	Let $F$ be a set endowed the notion of weak convergence as defined in Definition~\ref{weakconvDef}. Let $\bigcdot$ be a composition law on $F$ for which $F$ is total.	
  Let $(f_n)$ and $(g_n)$ be two sequences of elements of $F$. The composition is said to be regular on the weak closure $\overline{F}$ of $F$ if it satisfies the following conditions:
		\begin{itemize}
		    \item[1)] If sequences $(f_n)$ and $(g_n)$ are weakly convergent, then so is the sequence $(f_n\bigcdot g_n)_n$.
            \item[2)] If there exist weakly convergent sequences $(f_n')$ and $(g_n')$ in $F$ with $\widetilde{\lim} \, f_n = \widetilde{\lim} \, f_n'=f$ and $
		\widetilde{\lim} \, g_n = \widetilde{\lim} \, g_n'=g$, then $\widetilde{\lim}\, (f_n \bigcdot g_n) = \widetilde{\lim}\, (f_n'\bigcdot g_n')$. %= f \bigcdot g$.
		\end{itemize}
	\end{defn}

The conditions set forth by these definitions, if verified, permit the extension of the composition law to the weak closure. In order to alleviate the notation, from now on we write $\star$ for $\star_I$ wherever no confusion may arise.

 \begin{thm}[Mikusi\'nski \cite{Mikusiski1948SurLM}]\label{ThmMiku}
 Let $F$ be a set and $\bigcdot$ a composition law on $F$.
 If $\bigcdot$ is a regular operation as defined in Definition~\ref{regulardef}, then it is well defined on $\overline{F}$ so extends on it. In addition, if $\bigcdot$ is associative on $F$, so is it on its weak closure $\overline{F}$. 
 \end{thm}

Since $\mathcal{C}^\infty(I^2)$ has a convergence structure, our strategy for the $\star$~product of Definition~\ref{DefStarI} is thus to establish that:
 \begin{itemize}
 \setlength\itemsep{.1em}
     \item[i)] It is a composition law on $\mathcal{C}^{\infty}(I^2)$.
     \item[ii)] The set $\mathcal{C}^{\infty}(I^2)$ is total relatively to $\star$.
     \item[iii)] The $\star$~product is a regular operation on $\mathcal{C}^{\infty}(I^2)$.
     \end{itemize}
     While to identify elements of the weak closure 
     $\overline{\mathcal{C}^{\infty}(I^2)}$, we will prove that:
     \begin{itemize}
      \setlength\itemsep{.1em}
     \item[iv)] $\forall f \in \{\Theta, \delta^{(i)}\}_{i\in\mathbb{N}}$ there exists a weakly convergent sequence $(f_n)$ such that: 
     \subitem a) $\forall \varphi \in \mathcal{C}^{\infty}(I^2), \lim f_n \star \varphi = f \star \varphi$
     \subitem b) $f \star \varphi \in \mathcal{C}^{\infty}(I^2)$.
 \end{itemize} 

 \begin{rem}
Using Mikusi\'nski's construction, we are allowed to consider compositions of limits of sequences of elements of $F$ with elements of $F$--even if these limits lie outside of $F$ itself--on the condition that the result of the product be in $F$ (see Definition~\ref{weakconvDef}). If a product of a limit with an element of $F$ is found to be outside of $F$ this indicates one or more issues with the proposed definitions for $F$ and the composition. Observe now that for $\varphi\in  \mathcal{C}^{\infty}_c(\R^2)$ we may have $\Theta \star \varphi \notin \mathcal{C}^{\infty}_c(\R^2)$. 
This shows that we cannot use the usual space $F=\mathcal{C}^{\infty}_c(\R^2)$ in Mikusi\'nski's approach to formalize the $\star$~product. For this reason we must adopt $F=\mathcal{C}^\infty(I^2)$.
 \end{rem}
   
   We begin by considering the action of the $\star$~product on the space $\Cc$.
	\begin{align*}
		& \left \{ \begin{aligned} \Cc \times \Cc&\rightarrow \Cc, \\
			(f, \varphi) &\mapsto \int_I f(x, \tau) \varphi (\tau, y) d\tau.
		\end{aligned} \right.	
	\end{align*} 
	This map is well defined.
	\begin{lem}\label{intern}
		The $\star$~product defines a composition law on $\Cc$.
	\end{lem}

 \begin{proof}
Let $f$ and $g$ be in $\Cc$, we show that $f \star g$ is in $\Cc$ as well. Let $(x,y)$ belong to an open neighbourhood $\Omega$ of $I^2$ on which both $f$ and $g$ are smooth in the sense of Definition~\ref{DefCc}. Since $f$, $g \in \Cc$, $ \Vert f \Vert_{\infty} = \underset{(x,y)\in\Omega}{\sup} |f(x,y)|$ and $ \Vert g \Vert_{\infty}$ are finite and 
     \[
 | f(x,\tau) g(\tau, y) | \leq \Vert f \Vert_{\infty} \times \Vert g \Vert_{\infty}. 
     \]
     Thus, by the theorem on the continuous dependency of an integral on a parameter we conclude that, for all $(x,y) \in \Omega$, $f \star g$ is continuous as a function of $x$. In the same way, we show that $f\star g$ is continuous as a function of $y$ for $(x,y)\in \Omega$. Thus $f \star g$ is continuous on $\Omega$. 
     Using the same reasoning we establish the continuity of all derivatives of  $f \star g$ on  
     $\Omega$. Then $f \star g\in \mathcal{C}^\infty(I^2)$ and the $\star$~product is a composition law on $\mathcal{C}^\infty(I^2)$.
 \end{proof}

	\begin{lem}\label{lemTOT}
		The set $\Cc$ is total with respect to the $\star$ composition law.
	\end{lem}
	
	\begin{proof}
	  Let $f\in \mathcal{C}^\infty(I^2)$ be smooth on an open neighbourhood $\Omega$ of $I^2$. To facilitate the presentation of the proof we assume that $\Omega$ is of the form $\Omega_{\text{1D}}^2$ with $I\subset \Omega_{\text{1D}}$ open. By linearity of integration, it is sufficient to show that $f \star g = 0$ for all $g\in \Cc$ implies $f=0$ on $\Omega$. By assumption, we have 
		\[
		\int_I f(x,\tau) g(\tau,y) d\tau = 0,~~ \forall g \in \Cc.
		\]
		Choosing $g(x,y) = \overline{f(y,x)}$ for $(x,y)\in \Omega$ yields
		$$
		\int_I f(x,\tau) \overline{f(y,\tau)} d\tau = 0,
		$$    	
     	Then, for $y=x$ this gives
		$
		\Vert f(x, \cdot) \Vert ^2 = 0,
		$
	for all $x\in\Omega_{\text{1D}}$, where $f(x,\cdot) : y \mapsto f(x,y)$ and $\Vert \cdot \Vert$ refers to the norm associated to the Hermitian inner product on $L^2(\Omega_{\text{1D}})$. Then, since $\|.\|$ is a norm, $f(x,y) = 0$ as a function of $y\in\Omega_{\text{1D}}$. Furthermore, since $x\in\Omega_{\text{1D}}$ is arbitrary, we conclude that $f=0$ on $\Omega$.
	\end{proof}
	Given that $\Cc$ is endowed with the notion of uniform convergence, we have a corresponding notion of (uniform) weak convergence for its sequences in the sense of Definition~\ref{weakconvDef}. We denote $\Cb$ the weak closure of $\Cc$, that is, the set of all weak limits of sequences of elements of $\Cc$. 
\begin{lem}\label{lemReg}
		The $\star$ product is regular on $\Cb$.
	\end{lem}
 \begin{proof}
        Let $f\in\Cc$, by $\|f\|_\infty$ we designate the essential supremum of $f$ over $I^2$.
	 	Let $(f_n)_n$ and $(g_n)_n$ be two weakly convergent sequences with weak limits $f$ and $g$, respectively. Then, by definition,
  \begin{subequations}
	 \begin{align}
	 & \forall \varphi \in \Cc,\,\forall \epsilon > 0,\,\exists N_0 \in \N, \forall n \geq N_0,~\lVert f_n \star \varphi - f \star \varphi \rVert_\infty < \epsilon, \label{fnass}\\
	 & \forall \varphi \in \Cc,\, \forall \epsilon > 0,\,\exists N_1 \in \N, \forall n \geq N_1,~\lVert g_n \star \varphi - g \star \varphi \rVert_\infty < \epsilon,\label{gnass}
 \end{align}
  \end{subequations}
    and by construction $g\star \varphi\in \Cc$ so $f\star (g\star \varphi)\in \Cc$ is well defined. 
	We now  prove that $f_n\star g_n$ converges weakly. By Fubini's theorem, the $\star$~product is associative so,
	for all $\varphi\in\Cc$, 
 $$
	\lVert (f_n \star g_n) \star \varphi - f \star (g\star \varphi) \rVert_\infty = \lVert f_n \star (g_n \star \varphi) - f \star (g \star \varphi) \rVert_\infty,
	$$
	while the triangle inequality gives
	$$
	 \lVert f_n \star (g_n \star \varphi) - f \star (g \star \varphi) \rVert_\infty \leq \left\lVert f_n \star \big((g_n - g) \star \varphi\big) \right\rVert_\infty + \left\lVert (f_n-f) \star (g \star \varphi) \right\rVert_\infty.
	$$
    Given that $g \star \varphi \in \Cc $, by assumption (\ref{fnass}) there exists $N_0\in\mathbb{N}$ such that for $n\geq N_0$, $\lVert (f_n-f) \star (g  \star \varphi) \rVert_\infty < \epsilon$ for any $\epsilon>0$. Furthermore,
	$$
	 \left( f_n \star \big((g_n-g) \star \varphi\big)\right)(x,y) = \int_I f_n(x,\tau) \big((g_n-g)\star \varphi \big)(\tau,y) d\tau, 
	$$
  and therefore,
  $$
  \lVert f_n \star \big((g_n-g) \star \varphi\big) \rVert_\infty \leq \lVert (g_n-g) \star \varphi \rVert_\infty\, \left| \int_I f_n(x,\tau) d\tau\right|.
  $$
By Eq.~(\ref{gnass}), for any $\epsilon>0$ there exists $N_1\in\mathbb{N}$ such that for all $n\geq N_1$, $\lVert (g_n-g) \star \varphi \rVert_\infty < \epsilon $. Finally, noting that $\int_I f_n(x,\tau) d\tau = f_n \star 1$, it comes
 $$
\left| \int_I f_n(x,\tau) d\tau\right|\leq \,\lVert f_n \star 1 \rVert_\infty \leq \lVert f_n \star 1-f \star 1 \rVert_\infty + \lVert f \star 1 \rVert_\infty.
 $$
 Assumption (\ref{fnass}) guarantees the existence of $N_2$ with $n\geq N_2\Rightarrow \lVert (f_n-f) \star 1 \rVert_\infty  < \epsilon$ for any $\epsilon>0$. In addition, $f \star 1 \in \Cc$ and therefore $\lVert f \star 1 \rVert_\infty< \infty$.

 Choosing the same $\epsilon>0$ in all of the above arguments for convenience, there exists $\tilde{N}\geq \sup(N_0,N_1,N_2)$ such that $n\geq \tilde{N}$ implies
 $$
 \lVert (f_n \star g_n) \star \varphi - f \star (g \star \varphi) \rVert_\infty \leq \epsilon+\epsilon \big(\epsilon\, +\lVert f \star 1 \rVert_\infty\big)\xrightarrow[\epsilon\to 0]{} 0,
 $$
 which shows that sequence $f_n\star g_n$ converges weakly to $ h:=f\star g$ with $h\star \varphi=f\star (g\star\varphi)\in \Cc$ for all $\varphi \in \Cc$ as noted earlier.
 
 These results imply that both 1) and 2) of Definition~\ref{regulardef} are satisfied: 1) is entailed by the weak convergence of $f_n\star g_n$; while 2) follows from its convergence to $h$. Indeed, let $(f_n)_n, (f_n')_n, (g_n)_n \text{ and } (g_n')_n$ be sequences in $\Cc$. Suppose that 
	\(
	\widetilde{\lim}\, f_n = \widetilde{\lim}\, f_n'=f\) and \(
	\widetilde{\lim}\, g_n= \widetilde{\lim}\, g_n'=g.
	\)
	We have already established that
	$
	\widetilde{\lim} \, f_n \star g_n = h$
  and the same arguments lead to 
 $\widetilde{\lim} \,f_n' \star g_n'=h 
	$ as well, hence condition 2) is verified. 	
 \end{proof}
Since the $\star$~product is regular and associative on $\Cc$, by Theorem~\ref{ThmMiku} it extends to $\Cb$. In particular for any $f,g\in \Cb$, $f\star g$ is a well-defined element of $\Cb$, this set is total with respect to the $ \star$-product and the latter is associative over it. This proves Theorem~\ref{WellDefined}.
\end{proof}
\begin{rem}The $\star$~product is not commutative and we used only right $\star$~products in approaching elements of $\Cb$ but by regularity left $\star$~products are defined just as well.
\end{rem}

We  turn to proving Corollary~\ref{DinCb}. 
 \begin{proof}
 We adapt standard arguments from the theory of distributions. Let $\big(d_n\big)_n$ be a sequence of 
  nonnegative, smooth compactly supported functions of a \emph{single} variable with supports $I_n$, $I_n\subset \mathbb{R}$ compact, and such that,
 \begin{itemize}
     \item[1)] $\forall n \in \N, \int_{-\infty}^{+\infty} d_n(\tau) d\tau = 1$,
    % \item[1b)]    $\forall n \in \N,  \int_{-\infty}^{+\infty} d_n(\tau-y)d\tau =1$
     \item[2)] The size of $I_n$, denoted  $|I_n|$,  is such that  $|I_n| \underset{n \rightarrow +\infty}{\longrightarrow} 0$. 
 \end{itemize}
 Such a sequence exists: choose a sequence $I_n$ of compact sets with $0\in I_n$ and $|I_n|\to0$ as $n\to+\infty$. Then one can find a nonnegative smooth function with support exactly $I_n$, \cite[Prop. 2.3.4.]{krantz2002primer}. Dividing this function by the (non-zero) value of its integral over $I_n$ guarantees 1). Then define the sequence $(D_n)_n$ of functions of \emph{two} variables by $D_n(x,y):=d_n(x-y)$. Now we fix $I\subset \mathbb{R}$ compact and in the rest of this proof we consider only $\star_I$ products on this $I$. Observe that by construction $D_n\in\Cc$. %and $D_n$ has support $K_n=I_n\times I_n\subset I^2$. 
 Then, for $\varphi\in \Cc$ and $n\in\mathbb{N}$, the function $D_n\star \varphi-\varphi$ satisfies for all $(x,y)\in I^2$,
 \begin{align}
 \left\vert \big((D_n \star \varphi) - \varphi \big)(x,y) \right \vert &=\left\vert \int_I d_n(x-\tau)\varphi(\tau,y)d\tau - \varphi(x,y)\right \vert,\nonumber\\
  & \leq \int_{I} d_n(x-\tau) \left\vert \varphi(\tau,y)-\varphi(x,y)\right \vert d\tau \label{UpperBound}\\
     &\hspace{15mm}+ \left \vert \int_{I} d_n(x-\tau)\varphi(x,y)d\tau -\varphi(x,y) \right\vert. \nonumber
 \end{align}
Let $\tau':=x-\tau$ and suppose that $x$ is in the interior of $I$. Then observe that, 
\begin{equation}\label{interior}
\exists N\in \mathbb{N},\,\forall n>N,\quad\int_I d_n(x-\tau) d\tau = \int_{I_n}d_n(\tau')d\tau'=1.
\end{equation}
And so, for all $n>N$, 
%$\lim_{n\to +\infty}\int_I d_n(x-\tau) d\tau=1$. 
%For $x$ on a boundary of $I$, we set $\lim_{n\to +\infty}\int_I d_n(x-\tau) d\tau=1$ by continuity (note, this is equivalent to setting $\Theta(0)=1$, the convention which we follow, see p.~\pageref{Theta0}). As a consequence, for all $(x,y)\in I^2$, 
\begin{equation}\label{Simplification1}
\int_{I} d_n(x-\tau)\varphi(x,y)d\tau -\varphi(x,y)=0.
\end{equation}
In addition,  by the mean value inequality,
  \begin{align}\label{Simplification2}
      \int_{I} d_n(x-\tau) \left \vert \varphi(\tau,y)-\varphi(x,y) \right \vert d\tau, 
     & \leq \left  \Vert \partial_x\varphi \right \Vert_{\infty} \int_{I} d_n(x-\tau) \vert x-\tau \vert d\tau, 
     \end{align}
where $ \partial_x \varphi := \partial \varphi/ \partial x$ and $\Vert \partial_x\varphi  \Vert_{\infty}$ is determined over the open neighborhood of $I^2$ on which $\varphi$ is smooth. The precise choice for this open set does not matter here, as we need only $\Vert \partial_x\varphi  \Vert_{\infty}$ to be finite. 
By the argument presented in Eq.~(\ref{interior}) and thereafter we obtain, for all $n>N$,
$$
 \int_{I} d_n(x-\tau) \vert x-\tau \vert d\tau=  \int_{I_n} d_n(\tau') |\tau'|d\tau'\leq |I_n|\int_{I_n} d_n(\tau')d\tau'=|I_n|.
$$
%and similarly for $x$ on the boundary of $I$, by continuity. Therefore, 
Combined together, the results of Eqs.~(\ref{Simplification1}) and (\ref{Simplification2}), simplify Eq.~(\ref{UpperBound}) to  
\begin{align}\label{Conv0}
     \left \vert \left((D_n \star \varphi) - \varphi\right)(x,y) \right \vert 
     & \leq \left \Vert \partial_x\varphi \right \Vert_{\infty} \times |I_n| \underset{n \rightarrow +\infty}{\longrightarrow} 0.
 \end{align}
 As the upper bound here does  not depend $x$ nor $y$, the convergence to 0 is uniform. Considering now the boundary of $I$, observe first that for all $x$ in the interior of $I$, Eq.~(\ref{Simplification1}) implies that $\lim_{n\to+\infty}\int_I d_n(x-\tau )d\tau =1$ and so we  set $\lim_{n\to+\infty}\int_I d_n(x'-\tau )d\tau =1$ for $x'$ on the boundary of $I$ by continuity. This is equivalent to choosing $\Theta(0)=1$, the convention which we follow, see p.~\pageref{Theta0}. In turn this entails that, for $x'$ on the boundary of $I$, $\lim_{n\to+\infty}\left  \Vert \partial_x\varphi \right \Vert_{\infty} \int_{I} d_n(x'-\tau) \vert x'-\tau \vert d\tau=0$ so that Eq.~(\ref{Conv0}) is verified on the boundary of $I$.
  This establishes that $\delta \in \Cb$ and $\delta\star \varphi=\varphi$ for all $\varphi\in \Cc$. By regularity of the $\star$~product this remains true for all $\varphi\in\Cb$.
\begin{rem}
    This proves that $\delta$ is the left neutral element for $\star$. Considering $\varphi\star D_n$ a similar proof establishes that $\delta$ is also the right neutral element for $\star$.
\end{rem}

With the same sequences $(D_n)_n$ and $(d_n)_n$ defined above we have, 
 \begin{align*}
\left \vert \left(\partial_y D_n \star \varphi \right)(x,y) - \left(-\partial_y \varphi(x,y) \right) \right \vert 
& \leq \left[ d_n(x-\tau) \varphi(\tau,y)\right]_{I} \\
&\hspace{7mm}+ \left \vert \partial_y \varphi(x,y) - \int_{I} d_n(x-\tau) \partial_\tau \varphi(\tau,y)  d\tau  \right \vert .
%& \leq \left \vert \partial_y \varphi(x,y) - \int_{I_n} d_n(x-\tau) \partial_\tau \varphi(\tau,y)  d\tau \right \vert.
 \end{align*}
For any $x$ in the interior of $I$, there exists $N\in\mathbb{N}$ such that for all $n>N$, $\left[ d_n(x-\tau) \varphi(\tau,y)\right]_{I}=\left[ d_n(x-\tau) \varphi(\tau,y)\right]_{I_n}=0$. Moreover, since $\partial_{\tau} \varphi(\tau,y) \in \Cc$, the quantity $\partial_y \varphi(x,y) - \int_{I} d_n(x-\tau) \partial_\tau \varphi(\tau,y)  d\tau$ converges uniformly to 0. The situation on the boundary of $I$ is settled by continuity as presented above in details in the case of $\delta$. This shows that the Dirac delta derivative is an element of $\Cb$. By induction we obtain the same result for all its subsequent derivatives.

To alleviate notations we take $I=[a,b]$. We turn to proving that $\Theta\in\Cb$. Let $x \in I$ and define $h_n(x,y):=\int_{-\infty}^xd_n(\tau-y)d\tau$. Then $h_n \in \mathcal{C}^{\infty}(\R^2)$ and, for all $\varphi \in \mathcal{C}^{\infty}(\R)$, we have $\underset{n \rightarrow \infty}{\lim}  \langle \partial_{x} h_n, \varphi \rangle = \langle \delta , \varphi \rangle$. 
Furthermore, \(h_n \underset{n \rightarrow +\infty}{\longrightarrow} \Theta\) weakly. 
Indeed, 
\begin{itemize}
    \item[1)] If $x< y$,  there exists an integer $N$ such that for all $n>N$, $I_n \cap\, ]-\infty,x-y]=\emptyset$ and so $h_n(x,y)=0.$
    \item[2)] If $x> y$, there exists an integer $N$ such that for all $n>N$, $I_n \cap\, ]-\infty,x-y]=I_n$ and so $h_n(x,y)=\int_{I_n} d_n(\tau,y)d\tau=1$.
    \item[3)] If $y=x$, we set $\lim_{n\to+\infty} h_n(x,x)=1$ for all $x\in I$, consistently with the convention $\Theta(0)=1$.
\end{itemize}
Since $\forall n\in\mathbb{N}$, $h_n \in \Cc$ and $h_n$ is nonnegative (because $d_n$ is nonnegative),  $\forall \varphi \in \Cc$ we have,
\begin{align*}
\left \vert \big(h_n \star \varphi\big)(x,y) - \int_{a}^x \varphi(\tau,y)d\tau \right \vert& = \left \vert \int_{a}^x (h_n(x,\tau)-1) \varphi(\tau,y) d\tau + \int_x^b h_n(x,\tau) \varphi(\tau,y) d\tau \right \vert, \\
%&\hspace{10mm} = \left \vert \int_{a}^x \left(\int_{-\infty}^xd_n(\sigma-\tau)d\sigma-1\right) \varphi(\tau,y) d\tau + \int_x^b\left(\int_{-\infty}^xd_n(\sigma-\tau)d\sigma\right)  \varphi(\tau,y) d\tau \right \vert, \\
%&\hspace{10mm} \leq \int_{a}^x \left \vert \int_{-\infty}^x d_n(\sigma-\tau) d\sigma -1 \right \vert \left \vert \varphi(\tau,y) \right \vert d\tau + \int_x^b \left( \int_{-\infty}^x d_n(\sigma- \tau) d\sigma \right) \left \vert \varphi(\tau,y) \right \vert d\tau .
& \leq \int_{a}^x |h_n(x,\tau)-1|\, |\varphi(\tau,y)| d\tau + \int_x^b h_n(x,\tau)\, |\varphi(\tau,y)| d\tau.
\end{align*}
Now $\forall (x,y)\in I^2$, thanks to Lebesgue's dominated convergence theorem, we have
\begin{align*}
  \underset{n \rightarrow +\infty}{\lim} \int_{a}^x \left \vert h_n(x,\tau)-1\right \vert \vert \varphi(\tau,y) \vert  d\tau & = \int_{a}^x \left(   \underset{n \rightarrow +\infty}{\lim}  \left \vert h_n(x,\tau)-1 \right \vert  \right) \vert \varphi(\tau,y) \vert d\tau,  \\
  & =  0,
\end{align*}
and, 
\begin{align*}
  \underset{n \rightarrow +\infty}{\lim} \int_{x}^b h_n(x,\tau) \vert\varphi(\tau,y)\vert d\tau & = \int_{x}^b \left(  \underset{n \rightarrow +\infty}{\lim}h_n(x,\tau)\right) \vert \varphi(\tau,y)\vert d\tau,  \\
  & =  0.
\end{align*}
This shows that 
$
\left \Vert \big(h_n \star \varphi\big)(x,y) - \int_{a}^x\varphi(\tau,y) d\tau \right \Vert_{\infty} \underset{n \to +\infty}{\longrightarrow} 0$. Hence Heaviside theta is an element of $\Cb$ and $\Theta\star \varphi=\int_I \Theta(x-\tau)\varphi(\tau,y)d\tau = \int_a^x \varphi(\tau,y)d\tau$. Together with the results concerning $\delta$ and its derivatives, and multiplying point-wise by functions of $\Cc$ as needed, this establishes that $\mathcal{D} \subset \Cb$.
 \end{proof}
    
We prove Corollary \ref{CorStarR} regarding the $\star_\mathbb{R}$~product, starting with its failure to be well defined on $\mathcal{C}^\infty(\mathbb{R}^2)$.
\begin{proof}
 Suppose that the $\star_{\mathbb{R}}$~product on $\R$, defined by \[
\big(f\star_\mathbb{R} g\big)(x,y) := \int_{-\infty}^\infty f(x, \tau) g (\tau, y) d\tau,
\]
is a regular operation on $\mathcal{C}^\infty(\mathbb{R}^2)$. Now consider $(f_n)_n$ any sequence of functions of $\mathcal{C}^\infty(\mathbb{R}^2)$ converging to the constant function equal to $1$.  Then, regularity of the $\star_\mathbb{R}$ product would imply that $1 \star_{\mathbb{R}} 1 = \int_{-\infty}^{+\infty} 1 \times 1~ d\tau$ is well defined. 
Since it is not, the $\star_\mathbb{R}$ product is irregular and cannot be properly defined on $\overline{\mathcal{C}^\infty(\mathbb{R}^2)}$ and more generally on any unbounded subdomain of $\mathbb{R}^2$.
It is however possible to use $\star_\mathbb{R}$ on the smaller set $\D\subset \Cb$.
With the standard definition for the support of a distribution, maps $\tau\mapsto \delta^{(i)}(x-\tau)$, $\tau\mapsto\delta^{(j)}(\tau-y)$ and $\tau\mapsto\Theta(x-\tau)\Theta(\tau-y)$ are compactly supported. Then, for any $d,e\in \D$,  one can always find compacts $I\subset \mathbb{R}$ including these supports and consequently $d\star_\mathbb{R}e = d\star_I e$ for any $(x,y)\in\mathbb{R}^2$. 
Furthermore, thanks to these supports, we can extend the set $\D$ by considering the coefficients not in $\mathcal{C}^{\infty}(I^2)$ but in $\mathcal{C}^{\infty}(\R^2)$, the space of smooth functions on $\mathbb{R}^2$.
\end{proof}

Since the $\star$~product is associative, by linearity of the integral it is distributive with respect to the addition. In addition, by Corollary~\ref{DinCb} it has an identity element $1_\star \equiv\delta(x-y)$. Thus, we have the following corollary.
\begin{cor}\label{algebraD}
    $(\D,\star)$ is an algebra over $\mathbb{C}$ with unit. 
\end{cor}

\subsection{Inner and outer products, transpose}\label{InnOut}
 There are two natural injections of $\overline{\mathcal{C}^{\infty}(I)}$ into $\Cb$, which we call the left $\psi_{l}$ and right $\psi_r$ injections, with
 \begin{eqnarray*}
 \psi_l:~\overline{\mathcal{C}^{\infty}(I)}\hspace{-2mm}&\hspace{-20mm}\to \Cb,\hspace{8mm}&\psi_r:~\overline{\mathcal{C}^{\infty}(I)}\to \Cb,\\
 f\hspace{-2mm}&\mapsto \psi_l(f)(x,y) = f(x), \hspace{8mm} &\hspace{17.8mm}f\mapsto \psi_r(f)(x,y) = f(y).
 \end{eqnarray*}
 These injections permit the left $\star$~action of an element $g\in\Cb$ on any $f\in\overline{\mathcal{C}^{\infty}(I)}$ via its $\star$~action on $\psi_l(f)$,
 \begin{subequations}
 \begin{align}
 &\big(g \star \psi_l(f)\big) (x) = \int_{I} g(x,\tau)f(\tau) d\tau,\label{leftvec}\\
 &\big(\psi_l(f)\star g\big)(x,y) = f(x) \int_I g(\tau,y) d\tau.\label{leftouter}
\end{align}
\end{subequations}
Note that $\int_I g(\tau,y) d\tau\equiv 1\star g$ is well defined by construction of $\Cb$. Similarly, we have the right $\star$~action,
\begin{subequations}
 \begin{align}
 &\big(g \star \psi_r(f)\big) (x,y) = \int_{I} g(x,\tau) d\tau\,f(y),\label{rightouter}\\
 &\big(\psi_r(f)\star g\big)(y) = \int_I f(\tau) g(\tau,y) d\tau,\label{rightvec}
\end{align}
\end{subequations}
and finally for $h\in\overline{\mathcal{C}^\infty(I)}$, $\tilde{f}\in \mathcal{C}^\infty(I)$,
\begin{subequations}
 \begin{align}
 \big(\psi_r(h) \star \psi_l(\tilde{f})\big)  &= \int_{I} h(\tau)\tilde{f}(\tau) d\tau=\langle h,\tilde{f}\rangle,\label{inner}\\
 \big(\psi_l(h) \star \psi_r(\tilde{f})\big)(x,y) &= h(x)\int_I 1 d\tau \tilde{f}(y)=|I| h(x) \tilde{f}(y).\label{outer}
\end{align}
\end{subequations}
In Eq.~(\ref{inner}), the notation $\langle.,.\rangle$ designates an  inner product defining the action of linear forms $\mathcal{C}^\infty(I)\to \mathbb{C}$. Eq.~(\ref{inner}) establishes that the inner product on real-valued functions of $\mathcal{C}^{\infty}(I)$ is a $\star$~product and so is the action of the linear functional $h\in\overline{\mathcal{C}^\infty(I)}$ on a test function $f\in \mathcal{C}^\infty(I)$. Then, defining for $h,f,\in \overline{\mathcal{C}^{\infty}(I)}$, $\tilde{e}\in \mathcal{C}^{\infty}(I)$, $(h\otimes f)\star \psi_l(\tilde{e}) := h\langle f,\tilde{e}\rangle$, it appears that Eqs.~(\ref{leftouter}, \ref{rightouter}, \ref{outer}) are outer products.	
	\begin{defn}
		Let $g\in\Cb$ and $\tilde{f},\tilde{h}\in\mathcal{C}^\infty(I)$. The transpose $g^\mathrm{T}$ of $g$ is defined through $\langle \tilde{h}, g^\mathrm{T}\star \psi_l (\tilde{f}) \rangle:=\langle g\star \psi_l(\tilde{h}), \tilde{f} \rangle $.
	\end{defn}
	\begin{prop}
	    Let $g\in\Cb$. Then $g^\mathrm{T}(x,y)=g(y,x)$. 
	\end{prop}
	
	\begin{proof} 
Let $\tilde{f},\tilde{h}\in\mathcal{C}^\infty(I)$, then 
		\begin{align*}
			\langle g\star \psi_l(\tilde{h}), \tilde{f} \rangle &% = \int_I (g \star \tilde{h})(\tau) \tilde{f}(\tau) d\tau \\
			 = \int_I  \left( \int_I g(\tau, \sigma) \tilde{h}(\sigma) d\sigma \right) \tilde{f} (\tau) d\tau \\
			&= \int_I  \left( \int_I g(\tau, \sigma)  \tilde{f}(\tau) d\tau \right) \tilde{h}(\sigma) d\sigma\\% \text{, because } g \in \D\\
			%&= \int_I \left( g^\mathrm{T}\star \psi_l(\tilde{f}) \right)(\sigma) \tilde{h}(\sigma) d\sigma \\
			&= \langle \tilde{h}, g^\mathrm{T}\star \psi_l (\tilde{f}) \rangle.
		\end{align*}
	\end{proof}
	\begin{rem}
 For $g\in\overline{\mathcal{C}^\infty(I)}$, $\psi_l(g)=\psi_r(g)^\mathrm{T}$. The Hermitian conjugate $g^*$ of $g$ is as usual $g^*:=\bar{g}^\mathrm{T}$, where $\bar{g}$ denotes the complex conjugate of $g$. The usual inner product of complex-valued functions $\tilde{h},\tilde{f}\in\mathcal{C}^\infty(I)$ is recovered as 
 \( \langle \tilde{h},\tilde{f}\rangle=\psi_l(\tilde{h})^*\star \psi_l(\tilde{f})\)
 and similarly for the action of a linear functional $h\in\overline{\mathcal{C}^\infty(I)}$ on $\tilde{f}$,
  \[
 \langle h,\tilde{f}\rangle=\psi_l(h)^*\star \psi_l(\tilde{f}).
 \]
 In other terms, linear functionals send functions to the base field by $\star$~action. There arises the question of whether elements of $\Cb$ also act \emph{as linear functionals} on elements of $\Cc$ in the manner of a $\star$~product. The answer is both negative and positive. Negative because the action as linear functional $h\in \Cb$ on $\tilde{f}\in \Cc$ is
\begin{equation}\label{FunctionalActionCI2}
\langle h,\tilde{f}\rangle = \int_{I^2} h(x,y)\tilde{f}(x,y) dx dy, 
\end{equation}
which is not a $\star$~product of elements of $\Cb$. It is however induced by a higher dimensional $\star$~product defined for elements of $\overline{\mathcal{C}^\infty(I^4)}$ just as the $\star$~product on $\Cb$ induces the action of linear functionals in $\overline{\mathcal{C}^\infty(I)}$,
\[
h,f\in\overline{\mathcal{C}^\infty(I^4)},~(h\star f)(w,x,y,z):=\int_{I^2} h(w,x,\tau,\sigma)f(\sigma,\tau,y,z)d\tau d\sigma.
\]
This construction further extends to higher dimensions and will play a role when solving partial differential equations with $\star$~products. It exists because there is an injection of the linear functionals of a vector space into the endomorphisms of that space. Here the $\star$~product produces the composition of endomorphisms of $\Cc$, which through the injections $\psi_l$ and $\psi_r$, induces the action of linear functionals. This observation can be used to define a coherent umbral calculus on distributions, beyond the scope of this work.
\end{rem}

\subsection{Reductions to existing products}
In addition to the inner product on $\Cc$ and the action of linear functionals $\Cc\to \mathbb{C}$, the $\star$~product also induces a number of existing products.

\textit{Convolution.}
Consider two distributions $d,e\in \D$, $d(x,y)=\tilde{d}(x,y)\Theta+\sum_{i=0}^{+\infty} \tilde{d}_{i}(x,y) \delta^{(i)} $ and $e(x,y)=\tilde{e}(x,y)\Theta+\sum_{i=0}^{+\infty} \tilde{e}_{i}(x,y) \delta^{(i)}$ such that there exists functions $\tilde{D},\tilde{D}_i, \tilde{E}, \tilde{E}_i\in\mathcal{C}^\infty(I)$ with, for all $x,y,\in I$, 
\begin{align*}
\tilde{d}(x,y)=\tilde{D}(x-y),~\tilde{d}_i(x,y)=\tilde{D}_i(x-y),\\
\tilde{e}(x,y)=\tilde{E}(x-y),~\tilde{e}_i(x,y)=\tilde{E}_i(x-y).
\end{align*}
Then we may consistently define $D(x-y):=d(x,y)$ and $E(x-y):=e(x,y)$ and  
\begin{align*}
(d\star e)(x,y)&=\int_{-\infty}^{+\infty} d(x,\tau)e(\tau,y) d\tau,\\
& = \int_{-\infty}^{+\infty} D(x-\tau)E(\tau-y) d\tau,\\
&=(D\ast E)(x+y).
\end{align*}
In other terms a $\star$~product of two elements of $\D$ is a convolution if and only if these elements depend only on the difference between their two variables.

\textit{Volterra composition of the first kind.}
For two smooth functions $\tilde{f},\tilde{g}\in\Cc$, the Volterra composition of the first kind is defined as
\begin{equation*}
 \left(\tilde{f} \star_V \tilde{g}\right)(x,y) := \int_{y}^{x} \tilde{f}(x,\tau) \tilde{g}(\tau,y) d\tau,
 \end{equation*}
 for $x,y\in I$. As noted earlier Eq.~(\ref{Volt1}), a $\star$~product is a Volterra composition of the first kind if and only if the distributions multiplied are of the type $\tilde{f}(x,y)\Theta$.

\textit{Volterra composition of the second kind.}
For two smooth functions $\tilde{f},\tilde{g}\in C^\infty([0,1]^2)$, the Volterra composition of the second kind is defined as
\begin{equation*}
 \left(\tilde{f} \star_{V_{II}} \tilde{g}\right)(x,y) := \int_{0}^{1} \tilde{f}(x,\tau) \tilde{g}(\tau,y) d\tau.
 \end{equation*}
 This is a $\star_I$~product with $I=[0,1]$. Mikusi\'nski stated without proof that the Volterra composition of the second kind is a regular operation on the space of \emph{positive} smooth functions over $[0,1]^2$ \cite{Mikusiski1948SurLM}, now a consequence of Theorem~\ref{WellDefined}.

\textit{Pointwise product.} Let $\tilde{f},\tilde{g}\in \mathcal{C}^\infty(I)$, their ordinary pointwise product is $(\tilde{f}.\tilde{g})(x) =\tilde{f}(x)\tilde{g}(x)$. Now consider $f,g\in \D$ with $f=\tilde{f}(x)\delta$ and $g=\tilde{g}(x)\delta$. Then
$$
(f\star g)(x,y) = (\tilde{f}.\tilde{g})(x)\, \delta.
$$

\textit{Matrix product.} 
Let $x,y\in I$ and $\{x_i\in I\}_{0\leq i \leq N-1}$ with $x_0=x$ and $x_{N-1}=y$. For simplicity, suppose that the distance $|x_{i+1}-x_{i}|=\Delta x=1/N$ is the same for all $0\leq i\leq N-2$. This assumption is not necessary but alleviates the notation. 
 For $\tilde{f}\in \Cc$, we define a matrix $\mathsf{F}\in\mathbb{C}^{N\times N}$ with entries 
 $$
 \mathsf{F}_{i,j} := \tilde{f}(x_i,x_j)\Theta(x_i-x_j) .
 $$
  For example the matrix $\mathsf{H}\in \mathbb{C}^{N\times N}$ constructed from $\Theta$ is the \emph{lower triangular} matrix with $1$ on and under the diagonal. Any matrix $\mathsf{F}$ defined this way is lower triangular owing to the Heaviside step function. 
Constructing similarly another $\mathsf{G}\in\mathbb{C}^{N\times N}$ for $g(x,y)=\tilde{g}(x,y)\Theta(x-y)$, we observe that 
 \begin{equation*}
(\mathsf{F}.\mathsf{G})_{i,j}\times \Delta x=\sum_{x_j\leq x_k\leq x_i} \!\!\tilde{f}(x_i,x_k)\tilde{g}(x_k, x_j)\, \Delta x,
\end{equation*}
then,
\begin{equation*}
\lim_{\Delta x\to 0}(\mathsf{F}.\mathsf{G})_{0,N-1} \Delta x = \int_{y}^{x} \tilde{f}(x,\tau)\tilde{g}(\tau,y)d\tau\,\Theta(x-y)=(f\star g)(x,y).
  \end{equation*}
The above construction extends to all elements of $\D$ and maintains their algebraic structure: set $\delta$ in correspondence with $\mathsf{Id}_N/\Delta x$, $\mathsf{Id}_N$ being the identity matrix of size $N$; $\delta'$ is then in correspondence with the matrix inverse $\mathsf{H}^{-1}/\Delta x$ (because of \S\ref{MultIdent} below), $\delta''$ with $\mathsf{H}^{-2}/\Delta x^2$ and so on. 
This procedure provides a natural mean for numerical evaluations of $\star$~products via matrix-calculus. The accuracy of the resulting method is improved upon using other quadratures for integration (trapezoidal, Simpson etc.). An alternative approach consists of first expanding the smooth functions on a basis of Legendre polynomials and then multiply matrices of these coefficients to approximate $\star$~products \cite{pozza2023new}. 

The relation between $\star$~product and matrix product allows for an intuitive understanding of what $\D$ equipped with $\star$ is: a `continuum' version of the algebra of triangular matrices. Similarly the Fr\'echet Lie group of $\star$~invertible elements of $\D$,  constructed in Section \ref{FLgroup}, is a `continuum' version of the Borel subgroups formed by invertible triangular matrices. 
With this understanding, $\star$~products with elements of $\mathcal{C}^\infty(I)$ defined in \S\ref{InnOut} via $\psi_l$ and $\psi_r$ are the `continuum' versions of the column vector times matrix \eqref{rightvec} and matrix times line vector \eqref{leftvec} products, scalar product \eqref{inner} and outer products \eqref{leftouter}, \eqref{rightouter}, \eqref{outer}.

\textit{Derivation, integration and exponentiation.}
For any $\tilde{f}\in \Cc$, $\delta^{(n)}\star \tilde{f}$ and $\tilde{f}\star \delta^{(n)}$ are the $n$th derivatives with respect to the left and right variables of $\tilde{f}$, respectively. Furthermore, $\Theta\star \tilde{f}$ and $\tilde{f}\star \Theta$ are the left-variable and right-variable integrals of $\tilde{f}$, respectively; and more generally $\Theta^{\star n+1}\star \tilde{f}$ is equal to $(-1)^n/n!$ times the left-variable $n$th moment of $\tilde{f}$ while $ \tilde{f}\star \Theta^{\star n+1}$ is $1/n!$ times the right-variable $n$th moment of $\tilde{f}$. See also \S\ref{MultIdent} below. It follows from these facts that for a smooth function of a single variable $\tilde{h}\in \mathcal{C}^\infty(I)$
\[
\delta'\star \left(\exp(\tilde{h})\Theta\right) \star \left(\delta - \tilde{h} \Theta\right) = \delta,
\]
that is, exponentiation of a smooth function of a single variable is equivalent to taking a $\star$~resolvent of this function. This observation generalizes in many a ways, see \S\ref{MultIdent} for one way and here for another: if $\tilde{f}$ is replaced by an object that does not commute with itself in the $\star$~sense (as would typically be the case for a matrix of smooth functions), then its $\star$~resolvent yields a time-ordered (also known as path-ordered) exponential \cite{Giscard2015}. 

\begin{rem}
    Since $\overline{\mathcal{C}^{\infty}(I)}^{\otimes 2} \subset \overline{\Cc}$, the $\star$~product is also well defined on tensor products of distributions of one variable.
\end{rem}

	\section{Fr\'{e}chet Lie group on distributions} \label{FLgroup}
In this section we show that the $\star$~product induces the existence of a Fr\'{e}chet Lie group on a dense subset of $\D$, then show this Fr\'{e}chet Lie group is a subgroup of the automorphism group on $\Cc$. We begin with some multiplicative identities of $\D$, then provide this set with a metric, present existence and density results concerning $\star$~inverses and conclude with the Fr\'{e}chet Lie group structure.

From now on, we omit the $x-y$ arguments of the Heaviside $\Theta$ functions and Dirac deltas $\delta^{(i)}$ whenever
possible.

\subsection{Multiplicative identities}\label{MultIdent}

Since the product $f\star g$ reduces to a convolution when both $f$ and $g$ depend only on the difference between their variables, we immediately obtain the following (well known) identities involving some elements of $\D$. Firstly,
\begin{equation*}
\Theta\star \delta'=\delta,
\end{equation*}
which indicates that $\delta'^{\star-1}=\Theta$ and equivalently $\Theta^{\star-1}=\delta'=\delta^{(1)}$. Indeed, by Corollary~\ref{DinCb}, $\delta$ acts as the unit of the $\star$~product.
As a consequence, we may legitimately state $\Theta=\delta^{(-1)}$ and Definition~\ref{DefD} for $d\in\D$ is now
\begin{equation}\label{Dform}
		d(x,y)=\sum_{i=-1}^{+\infty} \tilde{d}_{i}(x,y) \delta^{(i)},
\end{equation}
where the sum starts at $i=-1$. 
Furthermore, we prove by Laplace transformation or directly by induction that, for $n\in\mathbb{N}\backslash\{0\}$,
\[
\big(\Theta^{\star n}\big)(x,y) = \frac{(x-y)^{n-1}}{(n-1)!}\Theta,
\]
meaning that $\Theta^{\star n}\propto \Theta$. Equivalently, this shows that all negative $\star$~powers of $\delta'$ are included in $\D$ and the sum in Eq.~(\ref{Dform}) above could run over $\mathbb{Z}$ just as well without changing $\D$. Conveniently,
\[
\delta^{(j)} = (\delta')^{\star j},
\]
which thus holds for $j\in\mathbb{Z}$, understanding that $\delta^{(-|j|)}:=(\delta')^{\star-|j|}=\Theta^{\star |j|}$.
We may therefore summarily write, for all $i,j\in\mathbb{Z}$,$$
\delta^{(i)}\star \delta^{(j)}=(\delta')^{\star i}\star (\delta')^{\star j}=\delta^{\star\, i+j}=\delta^{(i+j)} .
$$
As stated earlier, these results follow from the reduction of the $\star$~product to convolutions $\ast$. 
There are more general $\star$ identities when this is not the case. To present some examples of these, let $\tilde{f}\in\Cc$, and denote 
\(
\tilde{f}^{(k,\ell)}(\tau,\rho)
\)
the $k$th $x$-derivative and $\ell$th $y$-derivative of $\tilde{f}$ evaluated at $x=\tau$, $y=\rho$ with the conventions that $k=0$ or $\ell=0$ means no derivative is taken and $k=-1$ or $\ell=-1$ denotes integration. By associativity of the $\star$~product, $(\delta^{(k)}\star \tilde{f})\star \delta^{(\ell)}=\delta^{(k)}\star (\tilde{f}\star \delta^{(\ell)})=\tilde{f}^{(k,l)}$ is well defined.
Schwartz's results \cite[eqs. II,1; 5--7, p. 35]{schwartz1978} imply, for any $i,j\geq -1$ that \cite{GiscardPozzaConvergence},
 \begin{subequations}
\begin{align}
\delta^{(j)}\star \big(\tilde{f}(x,y)\delta^{(i)}\big) &= \tilde{f}^{(j,0)}(x,y)\delta^{(i)}+\sum_{k=1}^{j}\tilde{f}^{(j-k,0)}(y,y)\delta^{(i+k)},\label{leftaction}\\ 
\big(\tilde{f}(x,y)\delta^{(i)}\big)\star \delta^{(j)} &= (-1)^j\tilde{f}^{(0,j)}(x,y)\delta^{(i)}+\sum_{k=1}^{j}(-1)^{j+k}\tilde{f}^{(0,j-k)}(x,x)\delta^{(i+k)}.\label{rightaction}
\end{align}
\end{subequations}
Notice that the smooth function's partial derivatives are evaluated in $(y,y)$ and $(x,x)$ in the sums above, but not in the first term.
Finally, let us also present an example involving a $\star$~inverse that is not deducible from convolutions. For $\tilde{a},\, 
\tilde{b}\in\mathcal{C}^\infty(I)$,
\[
\left(\delta-\tilde{a}(x)\tilde{b}(y)\Theta\right)^{\star-1} =
 \delta+\tilde{a}(x)\tilde{b}(y)\exp\left(\int_y^{x} \tilde{a}(\tau)\tilde{b}(\tau) d\tau\right)\Theta, 
\]
which we verify directly by $\star$ multiplying the right-hand side with $\delta -\tilde{a}(x)\tilde{b}(y)\Theta$.
Further explicit results on $\star$~inverses and $\star$~multiplications are presented in \cite{GiscardPozza20201}, see also \S\ref{InvSection} below.

\begin{rem}
In general some care is required when representing elements of $\D$ and evaluating $\star$~products.
For instance, consider $\tilde{f}\in\Cc$ such that 
\[
\exists k\in\mathbb{N}:\, \forall j_1,j_2\in\mathbb{N}, j_1+j_2\leq k,\,\tilde{f}^{(j_1,j_2)}(x,x)=0.
\]
 Then $\tilde{f}(x,y)\delta^{(k)} \equiv 0_\D$ is null both as a linear functional $\Cc\to\mathbb{C}$ whose action is defined by the bracket of  Eq.~(\ref{FunctionalActionCI2}) and as endomorphism of $\Cb$, something which is not readily apparent from the notation alone. To further illustrate the notational difficulties consider calculating $\big(\tilde{f}(x,y)\delta'\big)\star \big(\tilde{g}(x,y)\delta'\big)$. Relying on the $\star$~action of the leftmost $\delta'$, Eq.~(\ref{leftaction}) yields
\begin{align}\label{LeftActionEx}
&\big(\tilde{f}(x,y)\delta'\big)\star \big(\tilde{g}(x,y)\delta'\big)=\\&\hspace{7mm}\tilde{f}^{(0,1)}(x,x)\tilde{g}(x,y)\delta'+\tilde{f}(x,x)\tilde{g}^{(1,0)}(x,y)\delta'+\tilde{f}(x,x)\tilde{g}(x,y)\delta^{(2)}.\nonumber
\end{align}
But we could equally well calculate this relying on the $\star$~action of the rightmost $\delta'$. Then Eq.~(\ref{rightaction}) gives
\begin{align}\label{RightActionEx}
&\big(\tilde{f}(x,y)\delta'\big)\star \big(\tilde{g}(x,y)\delta'\big)=\\&\hspace{7mm}-\tilde{f}^{(0,1)}(x,y)\tilde{g}(y,y)\delta'-\tilde{f}(x,y)\tilde{g}^{(1,0)}(y,y)\delta'+\tilde{f}(x,y)\tilde{g}(y,y)\delta^{(2)}.\nonumber
\end{align}
In spite of appearances the two results are equal: their actions as linear functionals $\mathcal{C}^\infty(I^2)\to \mathbb{C}$ are the same on any test function, and their compositions with any endomorphism of $\Cc$ is the same. This is because by \cite[eqs. II,1; 5–7, p. 35]{schwartz1978} for any smooth function $\tilde{h}$ and $k\in\mathbb{N}$,
\begin{align}\label{RepresentationsD}
\tilde{h}(x)\delta^{(k)}&=(-1)^k \big(\tilde{h}(y)\delta\big)^{(0,k)},\\
\tilde{h}(y)\delta^{(k)}&= \big(\tilde{h}(x)\delta\big)^{(k,0)},\nonumber
\end{align}
so defining e.g. $\tilde{h}(\tau):=\tilde{f}^{(0,1)}(x,\tau)\tilde{g}(\tau,y)$ and so on, one turns Eq.~(\ref{RightActionEx}) into Eq.~(\ref{LeftActionEx}).
\end{rem}

 \subsection{$\D$ is a Fr\'{e}chet space}\label{FrechetD}
The aim of this is to show that $\D$ is a metrizable space in the particular sense of Fr\'{e}chet spaces. 
We begin by defining a sequence of seminorms $(p_k)_k$ on $\D$ inducing a metric $\mathrm{d}$ on $\D$.

Let  $(K_k)_k$ be a sequence of compact spaces such that:\\[-1.8em] 
\begin{enumerate}
    \item $\forall i \in \N$, $ K_k \subset \overset{\circ}{K}_{k+1} $, where $\overset{\circ}{K}_{k+1}$ denotes the interior of the set $K_{k+1}$.
    \item \( \bigcup_{k=0}^{+\infty} K_k = \R^n,\)
    \item For any $K\subset \mathbb{R}^n$ compact, there exists \( k \in \mathbb{N} \) such that \( K \subset K_k\).
\end{enumerate}
Then $\mathcal{C}^{\infty}(\R^n)$ is a Fr\'{e}chet space when equipped with the sequence of seminorms 
\[
\tilde{p}_k(f) :=  \sup_{|\alpha| \leq k} \sup_{x \in K_k} |\partial^{\alpha} f(x)|=  \sup_{|\alpha| \leq k}  \Vert \partial^{\alpha} f \Vert_{\infty, K_k} ,
\]
where $\alpha$ is a multi-index. We generalise this construction to $\D$ by associating an element of $\D$ with a sequence of smooth functions.
	\begin{prop}
	Let \(d(x,y)=\sum_{i=-1}^{+\infty} \tilde{d_i}(x,y) \delta^{(i)}(x-y)\in \D
	\), $(K_k)_{k}$ a sequence of compacts of $\R^2$ as defined above, $\alpha=(\alpha_1,\alpha_2) \in \N^2$ a bi-index such that $\vert \alpha \vert = \alpha_1 + \alpha_2$ and let $\partial^\alpha := \frac{\partial^{\alpha_1}}{\partial x^{\alpha_1}} \frac{\partial^{\alpha_2}}{\partial y^{\alpha_2}}$.
	The sequence $(p_k)_{k\in\mathbb{N}\cup \{-1\}}$ of maps $p_k:\D\to\mathbb{R}^+$ defined by
	$$
	p_k(d) := \sup_{-1 \leq i \leq k} \tilde{p}_{k+1}(d_i) =\sup_{-1 \leq i \leq k} \sup_{|\alpha| \leq k+1} \Vert \partial^{\alpha} \tilde{d_i} \Vert_{\infty,K_{k+1}},
	$$
	is an increasing sequence of seminorms and the map $\mathrm{d}:\D^2\to\mathbb{R}^+$, defined for $d,e\in\D$ by 
	\[
 \mathrm{d}(d,e):=\sum_{k=-1}^{\infty} \frac{1}{2^{k+1}} \frac{p_{k}(d-e)}{1+p_{k}(d-e)},
	\]
 is a metric on $\D$.
 \end{prop} 
 \begin{proof}
Let $k \in \N \cup \{-1\}$. Since $\Vert \cdot \Vert_{\infty}$ is a norm on $\Cc$, every $p_k$ is a semi-norm on $\D$. The sequence $(p_k)_{k\in\mathbb{N}\cup \{-1\}}$ is an increasing sequence by construction, so it is an increasing sequence of seminorms. 
Therefore, $\mathrm{d}$ is a metric \cite{bony, zuilyqueff}.
 \end{proof}
 \begin{thm}
     Relatively to $\mathrm{d}$, $\D$ is a complete space and therefore a Fr\'{e}chet space. 
 \end{thm}
 \begin{proof}
	Let $(d_n) \in \D$ be a Cauchy sequence 
	\[
	\forall \epsilon > 0,\,\, \exists n_0 \in \N,\,\, \forall m,p \geq n_0\quad \mathrm{d}(d_m,d_p) < \epsilon, 
	\]
	that is,
	\[
	\sum_{k=-1}^{\infty} \frac{1}{2^{k+1}} \frac{\sup_{-1 \leq i \leq k} \sup_{|\alpha| \leq k+1} \Vert \partial^{\alpha} \tilde{d}_{m,i}-\partial^{\alpha} \tilde{d}_{p,i} \Vert_{\infty,K_{k+1}}}{1+\sup_{-1 \leq i \leq k} \sup_{|\alpha| \leq k+1} \Vert \partial^{\alpha} \tilde{d}_{m,i}-\partial^{\alpha} \tilde{d}_{p,i} \Vert_{\infty,K_{k+1}}} < \epsilon.
	\]
	In particular, for all $-1 \leq i \leq k$ fixed, 
	\[
	\sup_{|\alpha| \leq k+1} \Vert \partial^{\alpha} \tilde{d}_{m,i}-\partial^{\alpha} \tilde{d}_{p,i} \Vert_{\infty,K_{k+1}} \leq \sup_{-1 \leq i \leq k} \sup_{|\alpha| \leq k+1} \Vert \partial^{\alpha} \tilde{d}_{m,i}-\partial^{\alpha} \tilde{d}_{p,i} \Vert_{\infty,K_{k+1}}.
	\]
	Since  $x \mapsto \frac{x}{x+1}$ is a strictly increasing function,
	\[
	\frac{ \sup_{|\alpha| \leq k+1} \Vert \partial^{\alpha} \tilde{d}_{m,i}-\partial^{\alpha} \tilde{d}_{p,i} \Vert_{\infty,K_{k+1}}}{1+ \sup_{|\alpha| \leq k+1} \Vert \partial^{\alpha} \tilde{d}_{m,i}-\partial^{\alpha} \tilde{d}_{p,i} \Vert_{\infty,K_{k+1}}} < \frac{\sup_{-1 \leq i \leq k} \sup_{|\alpha| \leq k+1} \Vert \partial^{\alpha} \tilde{d}_{m,i}-\partial^{\alpha} \tilde{d}_{p,i} \Vert_{\infty,K_{k+1}}}{1+\sup_{-1 \leq i \leq k} \sup_{|\alpha| \leq k+1} \Vert \partial^{\alpha} \tilde{d}_{m,i}-\partial^{\alpha} \tilde{d}_{p,i} \Vert_{\infty,K_{k+1}}}.
	\]
	Hence, 
		\[
	\sum_{k=-1}^{\infty} \frac{1}{2^{k+1}} \frac{ \sup_{|\alpha| \leq k+1} \Vert \partial^{\alpha} \tilde{d}_{m,i}-\partial^{\alpha} \tilde{d}_{p,i} \Vert_{\infty,K_{k+1}}}{1+ \sup_{|\alpha| \leq k} \Vert \partial^{\alpha} \tilde{d}_{m,i}-\partial^{\alpha} \tilde{d}_{p,i} \Vert_{\infty,K_{k+1}}} < \epsilon.
	\]
	We set 
 \begin{align*}
 \tilde{\mathrm{d}}(f_m,f_p) &:= \sum_{k=-1}^{\infty} \frac{1}{2^{k+1}} \frac{\sup_{|\alpha| \leq k+1} \Vert \partial^{\alpha} \tilde{f}_{m,i}-\partial^{\alpha} \tilde{f}_{p,i} \Vert_{\infty,K_{k+1}}}{1+ \sup_{|\alpha| \leq k+1} \partial^{\alpha} \tilde{d}_{m,i}-\partial^{\alpha} \tilde{d}_{p,i} \Vert_{\infty,K_{k+1}}}  \\
 & := \sum_{k=0}^{\infty} \frac{1}{2^{k}} \frac{\tilde{p}_k(f_m-f_p)}{1 + \tilde{p}_k(f_m - f_p)}.  
 \end{align*}
	Since $(\mathcal{C}^{\infty}(I^2),\tilde{\mathrm{d}})$ is a  complete metric space, for all $i \in \N \cup \{-1\}$, there exists a function $\tilde{d_i} \in \mathcal{C}^{\infty}(I^2)$ such that
	\[
	\tilde{\mathrm{d}}(\tilde{d}_{m,i},\tilde{d_i}) \underset{m \to \infty}{\longrightarrow}0.
	\]
	Then, we define 
	\[
d(x,y):=\sum_{i=-1}^{+\infty} \tilde{d_i}(x,y) \delta^{(i)}(x-y),
\]
and it follows that
	\[
\mathrm{d}(d_{m},d) \underset{m \to \infty}{\longrightarrow}0.
\]
\end{proof}

 \subsection{Group structure on the $\star$~invertible elements of $\D$.}\label{InvSection}
 \begin{thm}
The set \(\mathop{\rm Inv}(\D)\) of $\star$~invertible elements of $\D$ is a dense subset of $\D$.
 \end{thm}

 \begin{proof}
     The proof relies on earlier constructive results by Giscard and Pozza \cite{GiscardPozza20201} regarding the existence of $\star$~inverses of certain elements of $\D$. We begin by recalling the necessary definitions. 
     \begin{defn}\label{separable}
         A smooth function $\tilde{f}(x,y)\in \Cc$ is said to be separable of finite order $s\in\mathbb{N}$ if and only if there exist ordinary smooth functions $\tilde{a}_i$ and $\tilde{b}_i$ with
$$
    \tilde{f}(x,y)=\sum_{i=1}^{s}\tilde{a}_i(x)\tilde{b}_i(y).
$$
     \end{defn}
     Then an element $d=\sum_{i=-1}^{+\infty} \tilde{d}_i(x,y) \delta^{(i)}\in \D$ is said to be separable if and only if all of its smooth coefficients $\tilde{d}^{(i)}$ are separable in the sense of Definition~\ref{separable}.  
\begin{lem}[Giscard, Pozza \cite{GiscardPozza20201}]\label{SmThetaInv}
Let $e\in \D$ be of the form $e(x,y)=\tilde{e}(x,y)\Theta(x-y)$ with a separable function $\tilde{e}\in\Cc$ that is not identically zero over $I^2$. Then $e^{\star-1}(x,y)$ exists everywhere in $I^2$ except possibly at a finite set of isolated points.
\end{lem}

Furthermore, a generic explicit (although involved) formula for the $\star$~inverse of such an element of $\mathcal{D}$ is presented in \cite{GiscardPozza20201}. This lemma  is sufficient to guarantee the $\star$~invertibility of more general elements of $\mathcal{D}$. Indeed, consider $d\in\mathcal{D}$ such that $\exists k\in\mathbb{N}\cup\{-1\}$ with $\tilde{d}_k\neq0$ and for all $j>k$ we have $ \tilde{d}_{j}=0\,$. We say that $d$ is of order $k$ and denote $\mathcal{D}^{(k)}$  the set of elements of $\mathcal{D}$ of order $k$.
Then, for any $k'\geq k$, $e:=d\star \Theta^{\star k'+1}$ is proportional to $\Theta$. If $e$ is separable, then by Lemma~\ref{SmThetaInv} it is invertible and, from there, so is $d$. It turns out that the separability of $d$ implies that $e$ is separable as well \cite{GiscardPozza20201}, herefore
     \begin{thm}[Giscard, Pozza \cite{GiscardPozza20201}]
  Let $d \in \D^{(k)}$ be separable. 
  Then the $\star$~inverse of $d$ exists and can be expressed as
  \[ d^{\star-1}(x,y) = \Theta^{\star(k+1)} \star e^{\star-1}(x,y), \]
  where $e(x,y):=\tilde{e}(x,y)\Theta(x-y)$ is separable and invertible by Lemma~\ref{SmThetaInv}.
  Furthermore, $d^{\star-1}$ is separable.
\end{thm}
Observe that a separable function of $\Cc$ is an element of $\mathcal{C}^\infty(I)\otimes\, \mathcal{C}^\infty(I)$ and this set is dense in  $\Cc$ by the Stone--Weierstrass theorem  since $I$ is compact. Then, by linearity, for any $k\in\mathbb{N}\cup\{-1\}$ the set of separable elements of $\D$ of order $k$ is a dense subset of $\D^{(k)}$. It follows that the set of separable elements of $\D$ of finite order is dense in the set of elements of $\D$ of finite order. We conclude by noting that for any $d\in\D$, the sequence $\big(d|_k:=\sum_{i=-1}^k \tilde{d}_i \delta^{(i)}\big)_k$ converges to $d$ with respect to $\mathsf{d}$ as $k\to\infty$. This implies that the set of elements of $\D$ of finite order is dense in $\D$.
 \end{proof}
 \begin{thm}\label{AutThm}
 $$\mathop{\rm Inv}(\mathcal{D})\subseteq \mathop{\rm Aut}(\mathcal{C}^{\infty}(I^2,\C)).$$
 \end{thm}
 \begin{proof}
The $\star$~product induces a group action of the group $(\mathop{\rm Inv}(\D),\star)$ on $(\Cc,+)$, which in particular satisfies 
\[
\forall d \in \mathop{\rm Inv}(\D),\,\, \forall \varphi, \psi \in \Cc,\quad d \star (\varphi + \psi)= d \star \varphi + d \star \psi.
\]
Consequently, for all $d \in \mathop{\rm Inv}(\D)$ the map $\varphi \mapsto d \star \varphi$ is an automorphism of $\Cc$ and $\mathop{\rm Inv}(\mathcal{D})\subseteq \mathop{\rm Aut}(\mathcal{C}^{\infty}(I^2,\C))$.
 \end{proof}
The results so far lead to the following theorem.
 \begin{thm}
    $(\mathop{\rm Inv}(\mathcal{D}), \star)$ is a Fr\'{e}chet Lie group. 
\end{thm}
\begin{proof}
    Since $\mathcal{D}$ is a Fr\'{e}chet space, it is sufficient (similarly to Banach Lie groups \cite{bonsall}) to show that 
    $$
    \forall d, e \in \mathcal{D},\,\, \forall k \in \N \cup \{-1\},\quad p_k(d \star e) \leq c_k \ p_k(d) \ p_k(e)
    $$
    with $c_k\in\mathbb{R}^+$ finite for $k$ finite.
    As the $\star$~product is linear and thanks to the triangular inequality, we need only show this for $d=\tilde{d}_p \delta^{(p)}$ and $e=\tilde{e}_m \delta^{(m)}$ with $m,p \in \N \cup \{-1\}$. 
     Firstly, we consider the particular case $p=m=-1$. Recall that $\delta^{(-1)}\equiv \Theta$. Let $k \in \N \,\cup\, \{-1\}$, we consider $(x,y) \in K_{k}^2$. Let $d(x,y)=\tilde{d}_{-1}(x,y) \Theta(x-y)$ and $e(x,y)=\tilde{e}_{-1}(x,y) \Theta(x-y)$. We have
    $$
    (d \star e) (x,y) = \int_y^x \tilde{d}_{-1}(x,\tau) \tilde{e}_{-1}(\tau,y) d\tau \Theta(x-y).
    $$
    Then, for $\alpha=(\alpha_1, \alpha_2) \in \N^2$, $|\alpha|=k+1$ 
    \begin{align*}
&\partial^{\alpha} \left(\int_{y}^x \tilde{d}_{-1}(x,\tau)\tilde{e}_{-1}(\tau,y) d\tau\right)=\\
&\hspace{5mm}\sum _{k=0}^{\alpha_1} \sum _{l=0}^{\alpha_1-1} \binom{\alpha_1}{l} \binom{\alpha_1-l-1}{k}
   \tilde{d}_{-1}^{(l,k)}(x,x) \tilde{e}_{-1}^{(\alpha_1-k-l-1,\alpha_2)}(x,y)\\
   &\hspace{10mm}-\sum _{k=0}^{\alpha_2} \sum _{l=0}^{\alpha_2-1} \binom{\alpha_2}{l} \binom{\alpha_2-l-1}{k}  \tilde{d}_{-1}^{(\alpha_1,\alpha_2-k-l-1)}(x,y)\tilde{e}_{-1}^{(k,l)}(y,y)\\
   &\hspace{17mm}+\int_y^x \tilde{d}_{-1}^{(\alpha_1,0)}(x,\tau) \tilde{e}_{-1}^{(0,\alpha_2)}(\tau,y) \, d\tau.
\end{align*}
Then we use $p_k(d)$ and $p_k(e)$ as upper bounds for the derivatives of $d$ and $e$, respectively. This leads to
\begin{align*}
p_k(d\star e)& = \text{sup}_{|\alpha|\leq k+1} \left\Vert \partial^{\alpha} \left( \int_y^x \tilde{d}_{-1}(x,\tau) \tilde{e}_{-1}(\tau,y) d\tau \right) \right\Vert_{\infty,K_{k+1}} \\
& \leq (3^{\max(\alpha_1,\alpha_2)} +|K_{k+1}|) p_k(d) p_k(e) \\
& \leq (3^{k+1} +|K_{k+1}|) p_k(d) p_k(e).
\end{align*}
Secondly, suppose without loss of generality that $m\neq -1$. Let $k \in \N \cup \{-1\}$ be a fixed integer, $K_{k+1}$ the corresponding compact as defined earlier; $d(x,y):=\tilde{d}_{p}(x,y) \delta^{(p)}(x-y)$ and $e(x,y):=\tilde{e}_{m}(x,y) \delta^{(m)}(x-y)$, $m\in \N$, $p \in \N \cup \{-1\}$. For $(x,y) \in K_{k+1}^2$, we have
\begin{align*} \tiny
 (d \star e)(x,y)&= \int_{-\infty}^{+\infty} \tilde{d}_{p}(x,\tau) \delta^{(p)}(x-\tau) \tilde{e}_m(\tau,y) \delta^{(m)}(\tau-y) d\tau \\
 &= (-1)^m \partial^{(0,m)} \left( \tilde{d}_{p}(x,y) \tilde{e}_m(y,y) \delta^{(p)}(x-y) \right)\\
 &= (-1)^m \sum_{j \leq m} \binom{m}{j} \partial^{(0,m-j)} \left(  \tilde{d}_{p}(x,y) \tilde{e}_m(y,y) \right) \delta^{(p+j)}(x-y).
\end{align*}
Then,
\begin{align*}
    p_k(d \star e)&= \sup_{-1 \leq p+j \leq k} \sup_{|\alpha| \leq k+1} \left \Vert \partial^{\alpha} \,(-1)^m\binom{m}{j} \partial^{(0,m-j)} \left(  \tilde{d}_{p}(x,y) \tilde{e}_m(y,y) \right) \right \Vert_{\infty,K_{k+1}}.
\end{align*}
Observe that we consider $p+m\leq k$ otherwise $p_k(d \star e) = 0$ trivially. Then $m \leq k-p$ and since $p \in \N \cup \{-1\}$, we get $m \leq k+1$. Consequently,
\begin{align*}
    p_k(d \star e)& \leq \sup_{-1 \leq j \leq k+1} \binom{m}{j} \sup_{|\alpha| \leq k+1} \left \Vert \partial^{(\alpha_1,\alpha_2+m-j)} \left(  \tilde{d}_{p}(x,y) \tilde{e}_m(y,y) \right) \right \Vert_{\infty,K_{k+1}} \\
    & \leq \sup_{-1 \leq j \leq k+1} \binom{m}{j} 2^{k+m-j} p_k(d) p_k(e) \\
    & \leq 3^{k+1} p_k(d) p_k(e).
\end{align*} 
\end{proof}

\begin{acknow}
M. R. and P. L. G. are funded by the French National Research Agency projects \textsc{Magica} ANR-20-CE29-0007 and \textsc{Alcohol} ANR-19-CE40-0006. We thank C. Miebach for his enlightening suggestions regarding Fr\'{e}chet Lie groups. We are grateful to the anonymous referee for his/her detailed remarks which have greatly helped improving this work.
\end{acknow}

%\bibliographystyle{plain}
%\bibliography{BiblioThese}

\end{document}